\documentclass[11pt,reqno]{amsart}
\usepackage[utf8]{inputenc}
\usepackage{amsmath,amsthm,url}
\usepackage{amssymb,bbm}
\usepackage{anysize}
\newtheorem{theorem}{Theorem}[section]
\newtheorem{corollary}[theorem]{Corollary}%[chapter]
\newtheorem{lemma}[theorem]{Lemma}%[chapter]
\newtheorem{proposition}[theorem]{Proposition}%[chapter]
\theoremstyle{definition}
  
\newtheorem{question}[theorem]{Question}
 
\newtheorem{definition}[theorem]{Definition}
\newcommand{\conv}{\mathop{\rm{conv}}}
\newcommand{\convSet}[1]{\mathop{\rm{conv}}\{#1\}}
\newcommand{\intr}{\mathop{\rm{int}}}
\newcommand{\R}{\mathbb{R}}
\newcommand{\Z}{\mathbb{Z}}

\newcommand{\calE}{\mathcal{E}}
\newcommand{\calM}{\mathcal{M}}
  
\begin{document}

\title[Lattice polytopes without interior lattice points] 
{Projecting lattice polytopes without interior lattice points}

\author{Benjamin Nill}

\address[Benjamin Nill]{
   Department of Mathematics, University of Georgia, Athens, GA 30602, USA
}
\email{bnill@math.uga.edu}

\author{G\"unter M. Ziegler}

\address[G\"unter M. Ziegler]{
   Inst.\ Mathematics, FU Berlin, Arnimallee 2, 14195 Berlin, Germany
}
\email{ziegler@math.fu-berlin.de}

\begin{abstract}
We show that up to unimodular equivalence 
there are only finitely many $d$-dimen\-sional lattice polytopes without interior lattice points 
that do not admit a lattice projection onto a $(d-1)$-dimensional lattice polytope without interior lattice points. 
This was conjectured by Treutlein. As an immediate corollary, we get a short proof of a recent result of Averkov, Wagner \& Weismantel, namely 
the finiteness of the number of maximal lattice polytopes without interior lattice points. Moreover, we show that in dimension four and higher 
some of these finitely many polytopes are not maximal as convex bodies without interior lattice points.
\end{abstract}

\maketitle

\section{Introduction and main results}

\subsection{Notation}

A \emph{lattice polytope} $P \subset \R^d$ is the convex hull of finitely many lattice points (in~$\Z^d$). 
We identify two lattice polytopes if they are \emph{unimodularly equivalent}, i.e., there is 
a lattice-preserving affine isomorphism mapping them onto each other. All lattice polytopes considered 
here will be \emph{$d$-polytopes}, i.e., full-dimensional, $\dim(P) = d$. The volume of a lattice polytope 
is always taken with respect to the given lattice (here, $\Z^d$), i.e., the volume of a fundamental parallepiped $[0,1]^d$ is equal to $1$.

In the geometry of numbers, convex bodies (compact convex sets) without lattice points are often called \emph{lattice-free}. 
Sometimes, this terminology is extended to convex bodies without \emph{interior} lattice points \cite{AWW10}. 
However, in the literature the term ``lattice-free polytopes'' had already been used to denote \emph{empty} lattice polytopes \cite{BK00} \cite{DO95} \cite{Kan99}, i.e., where 
the only lattice points are the vertices. In order to avoid these ambiguities we use in this note the following definition.

\begin{definition}
	A convex body $P \subset \R^d$ is \emph{hollow} if it does not contain any lattice points in its interior.
\end{definition}

A $0$-dimensional lattice polytope (a lattice point) is considered to be hollow. Important examples of hollow lattice polytopes are 
\emph{Cayley polytopes} (lattice polytopes of lattice width one), i.e., lattice polytopes whose vertices lie on two 
adjacent lattice hyperplanes, i.e.\ on two hyperplanes spanned by lattice points with no lattice points strictly between. 
Let $P \subset \R^d$ be a lattice polytope. We denote by a {\em lattice projection} a surjective affine-linear map 
$\phi$ from $\R^d$ to a vector space $V$ of dimension $m$ such that the kernel of $\phi$ is 
affinely generated by lattice points, or equivalently, such that $\phi(\R^d)$ is a lattice of rank $m$. By choosing 
an isomorphism $\phi(\R^d) \cong \Z^m$ we can identify $V \cong \R^m$, so $\phi(P) \subset \R^m$ is a lattice polytope. 
For instance, Cayley polytopes are precisely those lattice polytopes admitting a lattice projection onto $[0,1]$.

\subsection{The main theorem}
Lattice polytopes without interior points are of interest in geometry of numbers, optimization and 
Ehrhart theory, e.g., \cite{Sca85} \cite{Seb99} \cite{Bar02} \cite{BN07} \cite{AWW10}. The only hollow $1$-polytope is $[0,1]$. 
Hollow polygons are either isomorphic to $\convSet{(0,0),(2,0),(0,2)}$ or Cayley polytopes, see e.g. \cite{Rab89}. In dimension three, 
so far no complete classification exists, however recently some significant progress was made \cite{Tre08} \cite{Tre10} \cite{AWW09} \cite{AWW10}. 
A famous theorem by Howe \cite{Sca85} states that empty lattice $3$-polytopes are Cayley polytopes. 
In \cite{Tre08} Treutlein used this result to prove 
that hollow $3$-polytopes are either Cayley polytopes, admit a lattice projection onto $\convSet{(0,0),(2,0),(0,2)}$ 
or belong to a finite set of (not yet completely determined) exceptions. 
He conjectured that such a result should hold in any dimension. The 
confirmation of this conjecture is the main result of this paper.

\begin{theorem}Any hollow lattice $d$-polytope $P$ admits a lattice projection 
onto a hollow lattice $(d-1)$-polytope, except if $P$ belongs to one of finitely many exceptions.
\label{main}
\end{theorem}

The proof can be found in Section~\ref{proof}. It 
follows from combining the main results in two fundamental papers in the geometry of numbers: 
\cite{KL88} by Kannan and Lov\'asz, and \cite{Pik01} by Pikhurko. Roughly speaking, Kannan and Lov\'asz showed 
that if the volume of a hollow convex body is large enough, then it can be projected onto a lower-dimensional convex body where 
any lattice point lies arbitrarily close to the boundary. On the other hand, Pikhurko proved that such a lattice polytope has to be hollow.

\subsection{Maximal hollow lattice polytopes}

We say that a hollow lattice polytope is \emph{maximal hollow} if it is not properly contained in a larger hollow lattice polytope. 
As a corollary of Theorem~\ref{main} we get a finiteness result, which was recently 
proven by Averkov, Wagner \& Weismantel \cite[Theorem 2.1]{AWW10}:

\begin{corollary}
There are only finitely many maximal hollow lattice $d$-polytopes. Moreover, if a hollow lattice polytope 
is one of the exceptions in Theorem~\ref{main}, then it is contained in a maximal hollow lattice polytope.
\label{max}
\end{corollary}

\begin{proof}
Any lattice $d$-polytope $P$ that admits a lattice projection onto a hollow lattice polytope is necessarily also hollow. 
In particular, any such lattice polytope $P$ is strictly contained in a bigger hollow lattice polytope, so it cannot be maximal. 
On the other hand, if a hollow lattice polytope is not contained in a maximal one, 
then it is contained in an infinite inclusion-chain of hollow 
lattice polytopes. Hence by Theorem~\ref{main} one of these bigger hollow lattice polytopes admits a lattice projection 
onto a lower-dimensional hollow lattice polytope.
\end{proof}

Let us denote the finite list of maximal hollow lattice $d$-polytopes by $\calE_d$. The previous result shows that 
it is theoretically possible to completely `classify' all hollow lattice $d$-polytopes by determining  
$\calE_{d'}$ for $d' \leq d$. We have $\calE_1 = \emptyset$ and $\calE_2 = \{\convSet{(0,0),(2,0),(0,2)}\}$. 

\medskip

For $\calE_3$ an important partial result was achieved in \cite{AWW10} using a relaxed notion of maximality. 
For this, let us compare being maximal \emph{as a hollow lattice polytope} with being maximal \emph{as a hollow convex body}.
As was observed in \cite{Lov89}, a convex body is maximal as a hollow convex body 
if and only if it is a hollow polytope such that each facet contains a lattice point in its relative interior. 
Let us denote by $\calM_d$ the set of lattice $d$-polytopes which are maximal as hollow convex bodies. 
Of course, $\calM_d \subseteq \calE_d$. We observe that equality holds for $d \leq 2$. However, for $d \geq 4$ 
the inclusion $\calM_d \subsetneq \calE_d$ is strict:

\begin{theorem}
For $d \geq 4$, there are lattice $d$-polytopes that are maximal hollow lattice polytopes but not maximal hollow convex bodies.
\label{examples}
\end{theorem}

The proof is given in Section~\ref{examples-proof}. This leaves the case $d=3$, where it is believed but still open 
that there is no difference 
in these two notions of maximality.

\begin{question}\ {\rm 
Is $\calM_3 = \calE_3$~?
\label{q}}
\end{question}

In \cite{Tre08} \cite{Tre10} it was shown that $|\calM_3| \geq 10$. Recently, 
in \cite{AWW10} it was proven that $|\calM_3| = 12$. Therefore, confirming Question~\ref{q} would finish 
the classification of $\calE_3$.

\smallskip

The method of proof of Theorem~\ref{main} yields an upper bound on the volume of a lattice polytope 
in $\calE_d$. However, already for $d=3$ this yields a number with $117$ digits.
%\[\left(8(d-1)(8s+7)^{2^{2d-1}} + 1\right)^d\]
%\[\left(16\  15^{32} + 1\right)^3\]
Using the results of Pikhurko \cite{Pik01} we give a slightly more reasonable bound,
to be proved at the end of Section~\ref{proof}:

\begin{proposition} The volume of a lattice polytope in $\calE_3$ is at most $4106$.
\label{bound}
\end{proposition}

We remark that the maximal volume of a lattice polytope in $\calM_3$ is~$6$. 
It is likely that the methods of Treutlein \cite{Tre08} \cite{Tre10} 
could eventually lead to a sharpening of this bound or even a resolution of Question~\ref{q}.

\subsection{Lattice width}

One of the main motivation to study hollow lattice polytopes comes from the general interest in geometry of numbers and 
optimization in flatness and lattice width of convex bodies,  see e.g.~\cite{Bar02}. The \emph{lattice width} of a lattice polytope 
is defined as the infimum of $\max(u(P)) - \min(u(P))$ over all non-zero integer lattice directions $u$.

\begin{corollary} The maximal lattice width of $d$-dimensional hollow lattice polytopes equals the maximal lattice width of 
lattice polytopes in $\calE_{d'}$ for $d' \leq d$. 
Moreover, there are only finitely many $d$-dimensional hollow lattice polytopes whose lattice width is larger than that 
of any $(d-1)$-dimensional hollow lattice polytope.
\end{corollary}

\begin{proof} Let $P$ be a $d$-dimensional hollow lattice polytope with a lattice projection 
onto a $(d-1)$-dimensional hollow lattice polytope $P'$. Then the lattice width of $P$ is at most the lattice width of $P'$.
\end{proof}

The reader may compare this result with the conjecture  on empty lattice $d$-simplices in  \cite[Conj.~7]{HZ00}, which 
was recently proven for $d \leq 4$ \cite{BBBK09}.

\section{Proofs of Theorem~\ref{main} and Proposition~\ref{bound}}
\label{proof}

\noindent We will prove a slightly more general version of Theorem~\ref{main}. 
Throughout, let $P$ be a lattice $d$-polytope. We denote the interior of $P$ by $\intr(P)$. For a positive integer $s$, let us define as in \cite{Pik01}
\[I_s(P) := \intr(P) \cap s \Z^d.\]
We say that a lattice polytope $P$ is \emph{$s$-hollow}, if $I_s(P) = \emptyset$. 

\begin{theorem}
	Let $d,s\ge1$ be fixed. Then any $s$-hollow lattice $d$-polytope $P$ admits a lattice projection 
onto an \text{$s$-hollow} $(d-1)$-dimensional lattice polytope, except if $P$ belongs to one of finitely many exceptions. 
The volume of any such exceptional $s$-hollow lattice polytope is bounded by
\[s^d \left(8(d-1)(8s+7)^{2^{2d-1}} + 1\right)^d.\]
\label{main2}
\end{theorem}

In particular, as in the proof of Corollary~\ref{max}, we get the following  finiteness result, which was recently proven in \cite{AWW10}:

\begin{corollary}
	Let $d,s\ge1$ be fixed. Then there are only finitely many maximal $s$-hollow  lattice $d$-polytopes.
\end{corollary}

In order to prove Theorem~\ref{main2} we need the following notion, which was introduced in \cite{KL88}.

\begin{definition}
	A point $w \in \intr(P)$ is \emph{$\delta$-central} for some $\delta > 0$
	if for every $y \in P$ there is some $z \in P$ such that $z-w = \delta (z-y)$,
	i.e., $z-w = -\frac\delta{1-\delta}(y-w)$.
\end{definition}

Here is one of the main results in \cite{KL88}, Corollary (3.8), in a version for lattice polytopes.

\begin{theorem}[Kannan \& Lov\'asz 1988 \cite{KL88}] Let $P \subset \R^d$ be an $s$-hollow lattice $d$-polytope. 
Let $\Delta_1, \ldots, \Delta_d$ be real numbers with $0 = \Delta_0 < \Delta_1 < \cdots < \Delta_d \leq 1$. 
Then there is an $i \in \{0, \ldots, d-1\}$ and an $i$-dimensional subspace $U \subset \R^d$ affinely spanned by elements in $s \Z^d$ 
such that the lattice projection $\pi \;:\; \R^d \to \R^d/U$ satisfies the following two properties:
\begin{enumerate}
\item $\pi(P)$ does not contain points in $I_s(\pi(P))$ that are $\delta$-central with $\delta > \Delta_i$,
\item the volume of $\pi(P)$ (w.r.t. $\pi(\Z^d)$) is at most $s^{d-i}/(\Delta_{i+1}-\Delta_i)^{d-i}$.
\end{enumerate}
\label{theorem-kl}
\end{theorem}

\begin{proof} 
We want to apply \cite[Corollary (3.8)]{KL88} to the $s$-hollow lattice polytope $P$. 
Since this result is only formulated for convex bodies that contain no points in the lattice $s \Z^d$, 
we choose a fixed point $x \in \intr(P)$ and approximate $P$ by $P_t := t (P-x)+x$ for $0 < t < 1$. 
Let $b_1, \ldots, b_d$ be the reduced basis (see  \cite[Definition (3.3)]{KL88}) 
for the lattice $s \Z^d$ with respect to the centrally symmetric convex body $P_t - P_t = t (P-P)$. By definition, 
this basis is independent of the choice of~$t$. For $P_t$ it is shown in \cite[proof of Cor.~(3.8)]{KL88} that 
there is an $i \in \{0, \ldots, d-1\}$ such that the projection along the 
subspace spanned by $b_1, \ldots, b_i$ (for $i > 0$; respectively, along $\{0\}$ for $i=0$) satisfies the two desired conditions. 
Since we may assume that $i$ is the same for infinitely many $t$ arbitrarily close to $1$, the statement follows.
\end{proof}

Next, let us recall \cite[Theorem 4]{Pik01}. For this, we need the following definition:

\begin{definition}{\rm The \emph{coefficient of asymmetry} of an interior point $w$ of $P$ is defined as 
\[{\rm ca}(w, P) := \max_{|y|=1} \frac{\max\{\lambda \;|\; w + \lambda y \in P\}}{\max\{\lambda \;|\; w - \lambda y \in P\}}\]
}
\end{definition}

This notion is just a variant of the above definition of centrality: 
An interior point $w$ of $P$ with coefficient of asymmetry $c$ is $\frac{1}{c+1}$-central.

\begin{theorem}[Pikhurko 2001 \cite{Pik01}] Let $P' \subset \R^k$ be a $k$-dimensional lattice polytope with interior points in $s \Z^k$. 
Then there is a point $w \in I_s(P')$ with 
\[{\rm ca}(w,P') \leq 8k (8s+7)^{2^{2k+1}} - 1.\]
In particular, $w$ is $\delta$-central for
\[\delta = \frac{1}{8k (8s+7)^{2^{2k+1}}}.\]
\label{pik}
\end{theorem}

After these preparations, the proof of Theorem~\ref{main2} is quite straightforward.

\begin{proof}[Proof of Theorem~\ref{main2}]

We apply Theorem~\ref{theorem-kl} to the numbers 
\[\Delta_j :=  \frac{1}{8(d-j) (8s+7)^{2^{2(d-j)+1}} + 1}\]
with $j=1,\ldots, d$. Note that $0 < \Delta_1 < \cdots < \Delta_d = 1$.

In the case $i=0$, we observe from (2) that the volume of $\pi(P) = P$ is bounded by $s^d/\Delta_1^d$, 
a function in $d$. Hence, a result of Lagarias and the second author \cite[Thm.~2]{LZ91} implies that 
there are only finitely many lattice polytopes of at most this volume.

So, let $i \in \{1, \ldots, d-1\}$. We may assume that the $(d-i)$-dimensional lattice polytope 
$\pi(P)$ is not $s$-hollow. We apply Theorem~\ref{pik} for $P' := \pi(P)$ and $k := d-i$. Hence the point $w \in I_s(\pi(P))$ 
is $\delta$-central with $\delta > \Delta_i$, a contradiction to condition (1).
\end{proof}

\bigskip

Finally, the following sharpening of Theorem~\ref{pik} yields the proof of Proposition~\ref{bound}.

\begin{lemma}
Let $P'$ be a lattice polygon with interior lattice points. Then there is a point $w \in I_1(P')$ with 
\[{\rm ca}(w,P') \leq \frac{2}{0.124904} - 1\]
\label{lemma-improved}
\end{lemma}

\begin{proof}

Let $P'$ be a lattice polygon with interior lattice points. 
Let us choose a triangle $S \subset P'$ of maximal area. 
We may assume that the vertices $v_0, v_1 , v_2$ of $S$ are also vertices of $P'$ and that $P' \subseteq (-2) S + (v_0 + v_1 + v_2)$, 
see \cite{Pik01}. We consider three cases. 

\begin{enumerate}
\item If $S$ has no lattice points except its vertices, then it is unimodular equivalent 
to the unimodular triangle $\convSet{(0,0),(1,0),(0,1)}$. Therefore, $P'$ is contained in $(-2) S+ (1,1)$, which is hollow, a contradiction.  
\item If $S$ is hollow but has a lattice point in the interior of an edge, then by the classification of hollow lattice polygons 
we may assume that $S$ is of one of the following:
\begin{enumerate}
\item $S = \convSet{(0,0),(2,0),(0,2)}$. In this case, by going through the possible lattice subpolygons of $(-2) S + (2,2)$ one checks 
that there always is an interior lattice point $w$ of $P'$ with ${\rm ca}(w,P') \leq 3$.
\item $S = \convSet{(0,0),(k,0),(0,1)}$ for some $k \geq 2$. The facts that $P'$ is not hollow and that $S$ has maximal area imply that 
$P'= \convSet{(0,0),(k,0),(0,1),(k,-1)}$. In this case, $(0,\lfloor k/2 \rfloor)$ is an interior lattice point $w$ of $P'$ with ${\rm ca}(w,P') \leq 2$. 
\end{enumerate}
\item If $S$ has interior lattice points, then we can read off from  \cite[Table~1]{Pik01} (`guaranteed' lower bound for 
$\beta(2,1)$) that there is an interior lattice point $w$ 
with ${\rm ca}(w,S) \leq \frac{1}{0.124904}-1$. Therefore, \cite[Lemma 3]{Pik01} yields that 
${\rm ca}(w,P') \leq 2 \ {\rm ca}(w,S) + 2-1 = \frac{2}{0.124904} - 1 \approx 15.012$. 
\end{enumerate}

\end{proof}

\begin{proof}[Proof of Proposition~\ref{bound}]

Here $d=3$ and $s=1$. We define $\Delta_1 := 0.124903 / 2$, and choose $\Delta_1 < \Delta_2 < 1/3 < \Delta_3 = 1$. Then 
Lemma~\ref{lemma-improved} and the above proof of Theorem~\ref{main2} 
yields $1/(\Delta_1)^3 \approx 4105.55$ as the improved bound on the volume of a hollow lattice $3$-polytope $P$ 
in the case $i=0$.
\end{proof}

\section{Proof of Theorem~\ref{examples}}
\label{examples-proof}

\noindent In the following we describe for each $d\ge4$ a $d$-dimensional lattice polytope
$\Delta(d)_I$ that is \emph{hollow} (no interior lattice point),
is maximal as a hollow lattice polytope (it is not properly contained in a hollow lattice polytope),
but is not maximal as a hollow convex body (it is properly contained in the hollow simplex $\Delta(d)$).

\begin{definition}
For $d\ge3$ let $\Delta(d)$ be the hollow simplex given by
\label{delta-def}
\begin{eqnarray*}
	\Delta(d) &:=& \Big\{x\in\R^d : x_i\ge0 \textrm{ for } 1\le i\le d,\\
	          &  & \displaystyle\hphantom{\Big\{x\in\R^d : }
	\frac{x_1}2 + \frac{x_2}4 + \dots + \frac{x_{d-3}}{2^{d-3}} + \frac{x_{d-2}}{2^{d-2}} +
	\frac{x_{d-1}}{2^{d-1}-1} + \frac{x_d}{2^{d-1}+1+\alpha}\ \ \le\ \ 1\ \Big\}\\
	          & =& \conv
	\{ 0, 2e_1, 4 e_2, \dots, 2^{d-2} e_{d-2} , (2^{d-1}-1)e_{d-1}, (2^{d-1}+1+\alpha) e_d \},
\end{eqnarray*}
where $\alpha:=\frac1{2^{d-2}-1}>0$. Note that this is \emph{not} a lattice simplex for $d>3$, i.e.\ 
when $\alpha<1$. 
\end{definition}

\begin{theorem}\label{thm:Delta_I}
	For $d\ge4$ the integer hull $\Delta(d)_I:=\conv(\Delta(d)\cap\Z^d)$ is given by
\begin{eqnarray*}
  \Delta(d)_I & =& \Big\{x\in\R^d : x_i\ge0 \textrm{ for } 1\le i\le d,\\
              &  & \displaystyle\hphantom{\Big\{x\in\R^d : }  
	\frac{x_1}2 + \frac{x_2}4 + \dots + \frac{x_{d-3}}{2^{d-3}} + \frac{x_{d-2}}{2^{d-2}} +
	\frac{x_{d-1}}{2^{d-1}-1} + \frac{x_d}{2^{d-1}+1+\alpha}\ \ \le\ \ 1,\\
	      	  &  & \displaystyle\hphantom{\Big\{x\in\R^d : }
	\frac{x_1}2 + \frac{x_2}4 + \dots + \frac{x_{d-3}}{2^{d-3}} +
	\frac{2x_{d-2}+ x_{d-1} + x_d}{2^{d-1}+1} \qquad\ \ \le\ \ 1\ \Big\}\\
		          & =& \conv
		\{ 0, 2e_1, 4 e_2, \dots, 2^{d-2} e_{d-2} , (2^{d-1}-1)e_{d-1},\\
		      &  & \hphantom{\conv\{} 
		   (2^{d-1}+1) e_d, e_{d-1} + 2^{d-1} e_d, e_{d-2} + (2^{d-1}-1) e_d  \}.
\end{eqnarray*}	
This polytope is a $(d-3)$-fold pyramid over a triangular prism.

This lattice polytope is hollow. It is not contained properly in any larger $d$-dimensional
lattice polytope, but it is properly contained in $\Delta(d)$.
\end{theorem} 

\begin{proof}
	One first checks that the vertex description and the inequality description
	given in the definition of $\Delta(d)_I$ above are equivalent.
	
	If one takes the inequality system given for $\Delta(d)_I$ as the definition, then
	it is clear that  $\Delta(d)_I$ arises from $\Delta(d)$ by a hyperplane cut
	that cuts off the last, non-integral vertex $(2^{d-1}+1+\alpha) e_d$. The hyperplane cut 
	goes through the $d-3$ vertices $2e_1, 4 e_2, \dots, 2^{d-3} e_{d-3}$,
	while it cuts trough the tetrahedron spanned by the other $4$ vertices.
	This shows that $\Delta(d)_I$ is a $(d-3)$-fold pyramid over a tetrahedron
	with a vertex cut off, that is, a triangular prism.
	
	Using this combinatorial description one easily checks that the
	vertices of this polytope $\Delta(d)_I$ are indeed as listed in Theorem~\ref{thm:Delta_I}; thus, in particular,
	$\Delta(d)_I$ is a lattice polytope.

\medskip
	
	Next we check that $\Delta(d)_I$ is indeed the full integer hull of $\Delta(d)$,
	that is, that there are no integer points that would satisfy
\begin{eqnarray}
	   && x_i\ge0 \textrm{\quad for } 1\le i\le d,\label{eqn:0}\\[2pt] %	\nonumber \\[2pt]
	   && \displaystyle  
		\frac{x_1}2 + \frac{x_2}4 + \dots + \frac{x_{d-3}}{2^{d-3}} + \frac{x_{d-2}}{2^{d-2}} +
		\frac{x_{d-1}}{2^{d-1}-1} + \frac{x_d}{2^{d-1}+1+\alpha}\ \ \le\ \ 1, \label{eqn:1}\\
		      	  &  & \displaystyle
		\frac{x_1}2 + \frac{x_2}4 + \dots + \frac{x_{d-3}}{2^{d-3}} +
		\frac{2x_{d-2}\quad+\quad x_{d-1} \quad+\quad x_d}{2^{d-1}+1} \qquad\ \ >\ \ 1.\label{eqn:2}
\end{eqnarray}
	For the proof we rewrite the last two inequalities. We multiply (\ref{eqn:1}) by $-(2^{d-1}+1+\alpha)$ as
\begin{equation}
  - \sum_{i=1}^{d-3}(2^{d-1}+1+\alpha)\frac{x_i}{2^i}
  - (2^{d-1}+1+\alpha)\frac{x_{d-2}}{2^{d-2}}
  - (2^{d-1}+1+\alpha)\frac{x_{d-1}}{2^{d-1}-1}
  - x_d\ \ \ge\ \ -(2^{d-1}+1+\alpha) \label{eqn:3}
\end{equation}
    and we multiply the strict inequality (\ref{eqn:2}) by $2^{d-1}+1$,
\[
    \sum_{i=1}^{d-3}(2^{d-1}+1)\frac{x_i}{2^i}
  + 2 x_{d-2} + x_{d-1} + x_d\ \ >\ \ 2^{d-1}+1,   
\] 
    and then convert it into the equivalent (with respect to integer solvability)
    nonstrict inequality
\begin{equation}
    \sum_{i=1}^{d-3}(2^{d-1}+1)\frac{x_i}{2^i}
  + 2 x_{d-2} + x_{d-1} + x_d\ \ \ge\ \ 2^{d-1}+1+\frac1{2^{d-3}}. \label{eqn:4}  
\end{equation}
	Adding the two inequalities (\ref{eqn:3}) and (\ref{eqn:4}) we obtain
\[
	     \sum_{i=1}^{d-3} (-\alpha)\frac{x_i}{2^i} 
	   + \frac{-1-\alpha}{2^{d-2}} x_{d-2} + \frac{-2-\alpha}{2^{d-1}-1} x_{d-1} \ \ \ge\ \ \frac1{2^{d-3}}-\alpha.	  
	  	  %  - \sum_{i=1}^{d-3}(-\alpha)\frac{x_i}{2^i}
	  	  % + 2 x_{d-2} + x_{d-1} + x_d\ \ \ge\ \ \frac1{2^{d-3}}-\alpha.   
\]
Clearly, for $d\ge4$, that is when 
$\frac1{2^{d-3}}-\alpha=\frac1{2^{d-3}}-\frac1{2^{d-2}-1}>0 $, this cannot be satisfied while $x_1,\dots,x_{d-1}$ are nonnegative, by~(\ref{eqn:0}).	
(The inequality ``$x_d\ge0$'' is not needed for this.)
\medskip

	Finally, we show that $\Delta(d)_I$ is maximal as a hollow lattice polytope.
	For this, we check that out of the $d+2$ facets the first $d+1$ contain
	relatively interior points, which are $\mathbbm{1}-e_1,\mathbbm{1}-e_2,\dots,\mathbbm{1}-e_d$ for the first
	$d$ of them and $\mathbbm{1}$ (denoting the all-ones vector) for the next to last one.
	If an additional vertex were added to $\Delta(d)_I$ that is \emph{beyond} any of these (in the sense of Grünbaum
	\cite{Gr1-2} \cite{Z35}),
	then this automatically results in an interior lattice point. Thus any additional vertex
	must be beyond the last facet, but beneath all the other facets (or on their hyperplanes).
	However, there is no such integral point, because the first $d+1$ inequalities
	define $\Delta(d)$, and all integral points of this simplex lie in~$\Delta(d)_I$.
\end{proof} 

\begin{small}
\subsubsection*{Acknowledgements.}
The simplices $\Delta(d)$ with $\alpha=0$ appear in the diploma thesis \cite{dipl-Myrach} 
as examples of simplices that are maximal as hollow lattice simplices,
but not as hollow convex bodies; note that lattice simplices are maximal hollow lattice polytopes 
if and only if they are maximal hollow convex bodies. 
Thanks to Gregor Myrach for valuable discussions and to Jaron Treutlein for helpful comments. 
Lattice point computations using \texttt{polymake} \cite{polymake_lattice} were crucial help for Section~\ref{examples-proof}. 
This paper was finished when the first author stayed at Institut Mittag-Leffler during the program ``Algebraic Geometry with a view towards applications''. 
The second author was supported by DFG, Research Training Group ``Methods for Discrete Structures''. We thank the anonymous referee for helpful suggestions.
\end{small}

\bibliographystyle{siam}

\end{document}